\documentclass{llncs}
\usepackage{llncsdoc}

\usepackage{graphicx}
\usepackage{amsmath}
\usepackage{txfonts}
\usepackage{amssymb}
\usepackage{tikz}
\usepackage{comment}

\renewcommand{\P}{\mathcal{P}}

\newcommand{\up}[1]{#1^{\uparrow}}
\newcommand{\down}[1]{#1^{\downarrow}}
\newcommand{\ud}[1]{#1^{\uparrow\downarrow}}
\newcommand{\du}[1]{#1^{\downarrow\uparrow}}

\renewcommand{\emph}{\textbf}

\newcommand{\jty}{J^{\infty}}
\newcommand{\mty}{M^{\infty}}

\newcommand{\nomi}{\mathbf{i}}
\newcommand{\nomj}{\mathbf{j}}

\newcommand{\cnomm}{\mathbf{m}}
\newcommand{\cnomn}{\mathbf{n}}

\title{Categories: How I Learned to Stop Worrying and Love Two Sorts}

\author{Willem Conradie\inst{1} \and
Sabine Frittella\inst{2} \and
Alessandra Palmigiano\inst{1,2}
\and
Michele Piazzai\inst{2}
\and
Apostolos Tzimoulis\inst{2}
\and
Nachoem M. Wijnberg\inst{3}
}

\institute{Department of Pure and Applied Mathematics, University of Johannesburg\\
\email{wconradie@uj.ac.za}
\and
Faculty of Technology, Policy and Management, Delft University of Technology\\
\email{\{S.S.A.Frittella,A.Palmigiano,M.Piazzai,A.Tzimoulis-1\}@tudelft.nl}
\and
Amsterdam Business School, University of Amsterdam\\
\email{n.m.wijnberg@uva.nl,}
}

\begin{document}
\maketitle
  \begin{abstract}
    RS-frames were introduced by Gehrke as relational semantics for  substructural logics. They are two-sorted structures, based on RS-polarities with additional relations used to interpret modalities.
We propose an intuitive, epistemic interpretation of RS-frames for modal logic, in terms of categorization systems and agents' subjective interpretations of these systems. Categorization systems are a key to any decision-making process and are widely studied in the social and management sciences.

A set of objects together with a set of properties and an incidence relation connecting objects with their properties forms a polarity which can be `pruned' into an RS-polarity. Potential categories emerge as the Galois-stable sets of this polarity, just like the concepts of Formal Concept Analysis. An agent's beliefs about objects and their properties (which might be partial) is modelled by a relation which gives rise to a normal modal operator expressing the agent's beliefs about category membership. Fixed-points of the iterations of the belief modalities of all agents are used to model categories constructed through social interaction.\\

  \end{abstract}
%  \begin{keyword}
%  lattice-based modal logic, RS-frames, categorization theory, epistemic logic, formal concept analysis.
%  \end{keyword}

\noindent\textbf{Acknowledgement:}
The research of the first author has been made possible by the National Research Foundation of South Africa, Grant number 81309.
The research of the second to fourth author has been made possible by the NWO Vidi grant 016.138.314, by the NWO Aspasia grant 015.008.054, and by a Delft Technology Fellowship awarded in 2013.

\section{Introduction}
Relational semantic frameworks for logics algebraically captured by varieties of normal  lattice expansions\footnote{A {\em normal lattice expansion} is a bounded lattice endowed with operations of finite arity, each coordinate of which is either positive (i.e.\ order-preserving) or negative (i.e.\ order-reversing). Moreover, these operations  are either finitely join-preserving (resp.\ meet-reversing) in their positive (resp.\ negative) coordinates, or are finitely meet-preserving (resp.\ join-reversing) in their positive (resp.\ negative) coordinates.}  have been intensely investigated for more than three decades
\cite{bimbo2001four},
\cite{crapo1982unities},
\cite{dunn1990gaggle}, %17
\cite{duntsch2004relational},
\cite{galatos2007residuated}, %22
\cite{Gehrke}, %24
\cite{hartonas15modal},
\cite{jarvinen2005relational,kamide2002kripke},
\cite{Kurtonina,moshier2014}.
\cite{ploscica1995,Suz11}.
However,  none of these frameworks has gained the same pre-eminence and success as Kripke semantics. % has gained for the logics algebraically captured by varieties of normal distributive lattice expansions.
Indeed,  the extant proposals are regarded as significantly less intuitive than Kripke structures, especially w.r.t.\ their possibility to support the various established interpretations of modal operators (e.g.\ epistemic, temporal, dynamic), and hence doubts have been raised as to the suitability of these logics for applications. Various directions have been explored to try and cope with these difficulties, such as: (a) attempts to provide a conceptual justification to some of the distinctive features of these semantics (for instance, in  \cite{Gehrke}, a conceptual motivation has been given for the `two-sortedness' of the relational semantics for substructural logics introduced in the same paper  in terms of a duality between states and information quanta);
(b) recapturing the usual definition of the interpretation clause of modal operators in a generalized context
\cite{hartonas15modal,hartonas15order};
(c) improving the modularity of mathematical theories such as correspondence theory, to facilitate the transfer of results across different semantic settings. The latter direction has been implemented specifically for lattice-based logics in
\cite{CoCrPaZh16}, \cite{ConPal13,CoPa16}, and pursued more  in general in
\cite{CoCr14,ConPalSou},
\cite{CoFoPaSo15,CoGhPa14},
\cite{ConPal12,CoPaSoZh15,CoPaZh15,CoRo14},
\cite{FrPaSa14},
\cite{GMPTZ},
\cite{PaSoZh15b,PaSoZh15a}.

The contribution of the present paper pertains to direction (a): we propose categorization theory in management science as a concrete frame of reference for understanding the {\em RS-semantics} of lattice-based modal logic, and we argue that, when understood in this light, a natural epistemic interpretation can be given to the modal operators, which captures e.g.\ the factivity and positive introspection of knowledge. %\marginred{maybe we should cast our net a bit wider and talk about the possible connection with formal ontologies or description logic?}

Our starting point is the connection, mentioned also in \cite{Gehrke}, between RS-semantics and Formal Concept Analysis (FCA) \cite{Wille}. Namely, % Kripke frames for classical and intuitionistic modal logics are based on sets and posets respectively,
{\em RS-frames} for normal lattice-based modal logics  are based on {\em polarities}, that is, tuples $(A, X, \perp)$ such that $A$ and $X$ are sets, and $\perp\ \subseteq \ A\times X$. In FCA, polarities can be understood as {\em formal contexts}, consisting of objects (the elements of $A$) and properties  (the elements of $X$) with the relation $\perp$ indicating which object satisfies which property.  It is well known that any polarity induces a Galois connection between the powersets of $A$ and $X$, the stable sets of which form a complete lattice, and in fact, any complete lattice is isomorphic to one arising from some polarity. This representation theory for general lattices, due to Birkhoff, provides the polarity-to-lattice direction of the duality developed in \cite{Gehrke}, and is also at the heart of FCA. Indeed, the Galois-stable sets arising from formal contexts can be interpreted as {\em formal concepts}. One of the most felicitous insights of FCA is that concepts are  endowed by construction with a double interpretation: an {\em extensional} one, specified by the objects which are instances of the formal concept, and an {\em intensional} one, specified by the properties shared by any object belonging to the concept.

The second key step is the arguably natural  idea that categories and classification systems, as studied in social sciences and management science, are a very concrete setting of application of the insights of FCA. %In particular, categories and classification systems can be formalized in terms of  formal concepts and their lattices.
%Thus, we propose an understanding of  RS-semantics according to which  and that the formulas of lattice-based multi-modal logic and its relational semantics can be understood as the logic  and semantic environment of category labels.

Indeed, in social science and management science, {\em categories}   are understood as types of collective identities for broad classes e.g.\ of market products, organizations or individuals. Categorization theory recognizes categories as a key aspect of any decision-making process, in that they structure the space of options by defining the boundaries of meaningful comparisons between the available alternatives \cite{hsu2011typecasting}, \cite{wijnberg-gap}, \cite{wijnberg2011classification}.
Also, categories function as cognitive sieves, filtering out those features which are redundant or less essential to the decision-making, thus contributing to minimize the agents' cognitive efforts. Examples of categories are musical genres, which are widely applied as tools to compress and convey relevant information about a musical product to its potential audience.  Structuring information and decision-making along the faultlines of genres is so established a practice in the creative industries that genres have become the main way to structure competition as well as to create consumer group identity.

An aspect of categories which is very much highlighted in the categorization theory literature is that they never occur in isolation; rather, they arise in the context of {\em categorization systems} (e.g.\ taxonomies), which are typically organized in hierarchies of super- (i.e.\ less specified) and sub- (i.e.\ more specified) categories. This observation agrees with the FCA treatment, according to which concepts arise embedded in their concept lattice.

One of the open challenges in the extant literature is how to reconcile  the view on categories which defines them in terms of the objects (e.g.\ products) belonging to that category %or in terms of its {\em prototypical members},
with another view which defines categories in terms of the  features enjoyed by its members. The intensional and extensional perspectives on concepts brought about by FCA provide an elegant reconciliation of the two views on categories, which gives a second clue that the FCA perspective on categorization theory can be fruitful.

In recent years, a substantial research stream in social and management science explores the dynamic aspects of categorization \cite{hsu2011typecasting}, \cite{navis2010new} . For instance, {\em category emergence} investigates how new categories are created, either ex nihilo or through the recombination of existing ones, and how the interaction of relevant groups of agents, such as the media or the reviewers, plays a role in this process.
The aspect of social interaction is essential to understand how categories arise and are put to use:  although they can be seen to arise from factual pieces of information about the world (e.g., the products available in a given market and their features), a critical component of their nature cannot be reduced to factual information. In other words, categories are {\em social artifacts}, and reasoning about them requires a peculiar combination of factual truth, individual perception and social interaction. %One of the challenges registered by the extant literature is how to reconcile  one view on categories which defines them in terms of the objects (e.g.\ products) belonging to that category or in terms of its {\em prototypical members}, and another one which defines categories in terms of the ({\em essential}) features enjoyed by its members. The intensional and extensional perspectives on concepts brought about by FCA provide an elegant reconciliation of the two views on categories, and this is the first reason why we believe the FCA perspective on categorization theory is fruitful.

The main point of interest and the conceptual contribution of the present proposal concerns precisely the formalization of the  subjective  and social aspects of this emergence. Namely, we observe that the agents' subjective perspective on products and features can be naturally modelled by associating each agent with
%Thus, we propose an understanding of  RS-semantics according to which  and that the formulas of lattice-based multi-modal logic and its relational semantics can be understood as the logic  and semantic environment of category labels.We introduce
a binary relation $R\subseteq A\times X$ on the database $(A, X, \perp)$, which represents the subjective filters superimposed by each agent on the information of the database. That is,  for every product $a\in A$ and every feature $x\in X$, we read $aRx$ as `product $a$ has feature $x$ according to the agent'.  By general order-theoretic facts, these relations\footnote{Actually, those which are RS-compatible, cf.\ Definition \ref{Def:RSFrame}.} induce normal modal operators on the categorization system associated with the database. %Intuitively, the neighbourhood maps represent the subjective filters superimposed by each agent on the information of the database: %so for instance, neighbourhood maps can encode the topological distance of any product from the products perceived by the agent as prototypical, or collapse features of the database that are not relevant to the agent's own decision-making, or that the agent is not equipped to recognize.
These modal operators enrich the basic propositional logic of the categorization systems. In this enriched logical language, it is easy to distinguish between `objective' information (stored in the database), encoded in the formulas of the modal-free fragment of the language, and the agents' subjective interpretation of the `objective' information, encoded in formulas in which modal operators occur. This language is expressive enough to encode agents' beliefs/perceptions regarding other agents' beliefs/perceptions, and so on. %in a Harsanyi-type regression \cite{harsanyi2004games}.
Again, this makes it possible to define {\em fixed points} of these regressions, similarly to the way in which common knowledge is defined in classical epistemic logic \cite{fagin2003reasoning}. Intuitively, these fixed points represent the stabilization of a process of social interaction; for instance, the consensus reached by a group of agents regarding a given category. Clearly, market dynamics are bound to create further destabilization, necessitating a new round of interaction in order to establish a new equilibrium. Further directions will be to generalize the framework of dynamic epistemic logic \cite{baltag1998logic} to the setting outlined in the present paper, and further develop the theory of lattice-based mu-calculus initiated in \cite{CoCrPaZh16}.

%The picture sketched so far presents a rich environment which brings together several layers of information key to category emergence: the factual, subjective and social. In this environment, many subtle social dynamics can be modelled, and their consequences on agents decision-making, both individual and collective.

%We do not prove any new results, but provide a new conceptual framework. This is beneficial for the theory of normal lattice expansions modal logic, since it helps to make sense of something mathematically grounded but conceptually very mysterious; it is beneficial for categorization theory, since this interpretation helps to reconcile aspects of the theory which are at odds with each other in the extant literature (such as the intentional vs extensional view on categories and the interplay of factual, individual and social aspects in the social construction of categories).

%Finally, although this paper does not provide any new mathematical results, it motivates conceptually the research on nontrivial mathematical results, such as the theory of lattice based mu-calculus \cite{} (this because we can construct the real categories as a common-knowledge operators, and hence fixed points).
{\it Structure of the paper.} In Section \ref{sec:prelim}, we collect the necessary definitions and basic facts about RS-semantics. In Section \ref{sec:categorization}, we discuss how the mathematical environment introduced in the previous section can be understood using categories and categorization systems as the framework of reference. In particular, we show how normal modal operators on lattices can support  an epistemic interpretation.   In Section \ref{sec:social-construction}, we build on the epistemic interpretation of the modal operators, and introduce a common knowledge-type construction to account for a view of categories as the outcome of social interaction. In Section \ref{CCL} we collect our conclusions. More technical background is relegated to  Appendix \ref{Appendix}, while the proofs of some technical lemmas can be found in Appendix \ref{Appendix:TechLemmas}.

\section{Preliminaries}
\label{sec:prelim}
%In this section we report on the results presented in \cite{VanGool:MThesis}.

In this section we recall some preliminaries on perfect lattices, RS-polarities, generalized Kripke frames and formal concept analysis. We will assume familiarity with the basics of lattice theory (see e.g.\ \cite{DaveyPriestley2002}).
\subsection{Perfect lattices}
\label{ssec:Prelim:Lattices}
A bounded lattice $\mathbb{L} = (L, \wedge, \vee, 0, 1)$ is  {\em complete} if all subsets $S\subseteq L$ have both a supremum $\bigvee S$ and an infimum $\bigwedge S$. %It is easy to see that complete lattices are always bounded.
An element $a$ in  $\mathbb{L}$ is  {\em completely join-irreducible} if, for any $S\subseteq\mathbb{L}$, $a = \bigvee S$ implies $a \in S$. {\em Complete meet-irreducibility} is defined order-dually. The sets of completely join- and meet-irreducible elements of $\mathbb{L}$ are denoted by $J^{\infty}(\mathbb{L})$ and $M^{\infty}(\mathbb{L})$, respectively.

A complete lattice $\mathbb{L}$ is called {\em perfect} if it is join-generated by its completely join-irreducibles, and meet-generated by its completely meet-irreducibles. That is, $\mathbb{L}$ is perfect if for any $u\in \mathbb{L}$, we have
%
%\[
$\bigvee \{j\in J^{\infty}(\mathbb{L})\mid j\leq u\}=u= \bigwedge\{m\in M^{\infty}(\mathbb{L})\mid u\leq m\}$.
%\]

\subsection{Polarities and Birkhoff's representation theorem}
\label{ssec:Prelim:RS}

\begin{definition}
\label{def:polarities}
A {\em polarity} is a triple $\mathbb{P} = (A,X,\perp)$ where $A$ and $X$ are sets, and $\perp\ \subseteq\ A \times X$ is a relation.
 For every polarity $\mathbb{P}$, we define the functions $\up{(\cdot)}$ (upper) and $\down{(\cdot)}$ (lower)\footnote{In what follows, we abuse notation and write $\up{a}$ for $\up{\{a\}}$ and $\down{x}$ for $\down{\{x\}}$ for every $a\in A$ and $x\in X$.} between the posets $(\P(A), \subseteq)$ and $(\P(X) , \subseteq)$, as follows:
\begin{center}
for $U\in\P(A)$, let $\up{U} := \{
x\in X \mid \forall a( a\in U \rightarrow a\perp x)\}$,
\\
for $V\in \P(X)$, let
$\down{V} := \{ a\in A \mid \forall x(x\in V\rightarrow a\perp x) \}$.
\end{center}
\end{definition}
The maps $\up{(\cdot)}$ and $\down{(\cdot)}$  form a \emph{Galois connection} between $(\P(A), \subseteq)$ and $(\P(X) , \subseteq)$, i.e.\ $V\subseteq \up{U}$ iff $U \subseteq \down{V}$ for all $U \in \P(A)$ and $V \in \P(X)$. Well-known consequences of this fact are: the composition maps $\ud{(\cdot)}: = \down{(\cdot)}\circ \up{(\cdot)}$ and $\du{(\cdot)}: = \up{(\cdot)}\circ \down{(\cdot)}$ are closure operators on $(\P(A), \subseteq)$ and $(\P(X), \subseteq)$, respectively;\footnote{Recall that a closure operator on a poset $(S, \leq)$ is a map $f: S \rightarrow S$  which is extensive ($\forall a \in S [a \leq f(a)]$), monotone ($\forall a, b \in S [a \leq b \Rightarrow f(a) \leq f(b)]$) and idempotent ($\forall a \in S [f(a) = f(f(a))]$).}  %A subset $U \in \P(X)$ is \emph{Galois-closed} if $\low(\up(U)) = U$. Similarly $V \in \P(X)$ is Galois closed if $\up(\low(V)) = V$. %We define the map $\cl_R: \P(X) \rightarrow \P(X)$ as the closure operator $\low_R \up_R$ which sends every subset of $X$ to its Galois closure.
The set of all {\em Galois-stable} subsets of $A$ (i.e.\ those $U \in \P(A)$ such that $\ud{U} = U$) forms a complete sub-semilattice of $(\P(A), \bigcap)$, which we denote by $\mathbb{P}^+$;\footnote{Likewise, The set of all {\em Galois-stable} subsets of $X$ (i.e.\ those $V \in \P(X)$ such that $\du{V} = V$) forms a complete sub-semilattice of $(\P(X), \bigcap)$.} since it is complete, the semilattice $\mathbb{P}^+$ is in fact a lattice, where meet is set-theoretic intersection and join is the closure of the set-theoretic union. If fact, Birkhoff showed that every complete lattice is isomorphic to  $\mathbb{P}^+$ for some polarity $\mathbb{P}$. This lattice can be identified with the lattice of {\em concepts} arising from $\mathbb{P}$ (this terminology comes from Formal Concept Analysis), i.e.\ tuples $(C, D)$ s.t.\ $C\subseteq A$, $D\subseteq X$ and $\down{D} = C$ and $\up{C} = D$.\footnote{\label{footnote:intension} Sometimes $C$ and $D$ are referred to as the {\em extension} and the {\em intension} of a concept, respectively.}  Concepts (resp.\ Galois stable subsets of $X$ and of $A$) can be characterized as (members of) tuples $(\ud{U}, \up{U})$ and $(\down{V}, \du{V})$ for any $U\subseteq A$ and $V\subseteq X$.

%Before moving on to discuss how Birkhoff's representation theorem specializes to a duality for perfect lattices,
Let us conclude the present subsection by introducing some notation and showing some useful facts. Polarities $(A, X, \perp)$ induce `specialization pre-orders' on $A$ and $X$ defined as follows:  $x\leq y$ %($x$ `implies' $y$)
iff $\forall a (a\perp x\rightarrow a\perp y)$ for all $x, y\in X$, and $a\leq b$ %($a$ is at least as specified as $b$)
iff $\forall x (b\perp x\rightarrow a\perp x)$ for all $a, b\in A$. Clearly, $\leq \circ \perp\circ \leq\  \subseteq \ \perp$.
For every $b\in A$ and $z\in X$, let $z{\uparrow}: = \{x\mid z\leq x\}$, and %$a{\uparrow}: = \{b\mid a\leq b\}$ $x{\downarrow}: = \{y\mid y\leq x\}$, and
$b{\downarrow}: = \{a\mid a\leq b\}$. %\footnote{Notice the difference between e.g.\ $\up{a}\in \P(X)$,  and $a{\uparrow}\in \P(A)$.}
The proofs of the following lemma and corollary can be found in Appendix \ref{Appendix:TechLemmas}.

\begin{lemma}\label{Lemma:Galois:Stable} $z{\uparrow}$ and $b{\downarrow}$ are Galois-stable for all $b\in A$ and $z\in X$.\end{lemma}

%Willem: Moved to the appendix
%\begin{proof} We only prove the part concerning $z$. Let $x\in z{\uparrow}^{\downarrow\uparrow}$, and let us show that  $z\leq x$. That is,  let us fix $a$ such that $a\perp z$, and  show that $a\perp x$.
%Since $\perp\circ \leq\  \subseteq \ \perp$,  from $a\perp z$ it follows that $\forall y (z\leq y\rightarrow a\perp y)$, which means that $a\in z{\uparrow}^{\downarrow}$. Since by assumption $x\in z{\uparrow}^{\downarrow\uparrow}$, this implies that $a\perp x$, as required.
%\end{proof}

\begin{corollary}\label{Cor:Up:Down}
$\du{z} = z{\uparrow}$ and $\ud{b} = b{\downarrow}$ for all $b\in A$ and $z\in X$.
\end{corollary}
%Willem: Moved to the appendix
%\begin{proof}
%Since $z{\uparrow}$ is Galois-stable and contains $z$ and, by definition, $\du{z}$ is the smallest such set, $\du{z}\subseteq z{\uparrow}$. For the converse inclusion, let $z\leq y$  and $a\perp z$. As $\perp\circ \leq\ \subseteq\ \perp$, this implies $a\perp y$, which shows that $y\in \du{z}$, as required.
%\end{proof}
Summing up, the concepts generated by each $a\in A$ and $x\in X$ are $(a{\downarrow}, \up{a})$ and $(\down{x}, x{\uparrow})$ respectively.
 \subsection{RS-polarities and dual correspondence for perfect lattices}
 As mentioned early on,  every complete lattice is isomorphic to  $\mathbb{P}^+$ for some polarity $\mathbb{P}$.
 When specializing to distributive lattices and Boolean algebras, the well-known dualities obtain between set-theoretic structures and {\em perfect} algebras. In particular, perfect distributive lattices are dual to posets, and perfect (i.e.\ complete and atomic) Boolean algebras are dual to sets. %Or, even more pertinent from a logical perspective, distributive lattices with (complete) operators are dual
The question then arises: which polarities are dual to perfect lattices? The answer was given by Gehrke in \cite{Gehrke}, where the so-called reduced and separated polarities, or {\em RS-polarities}, have been characterized as duals to perfect lattices, by rephrasing in a model-theoretic way the duality for perfect lattices given in \cite{DGP05}. In what follows, we will recall what it means for a polarity to be reduced and separated, and briefly explain how these two properties guarantee the perfection of the dual lattice. First, the route from perfect lattices to polarities is given by the following definition:

\begin{definition}
 For every perfect lattice $\mathbb{L}$, the {\em polarity associated with $\mathbb{L}$} is  the triple $\mathbb{L}_+ := (J^\infty (\mathbb{L}),M^\infty(\mathbb{L}),\perp_+)$ where $\perp_{+}$ is  the lattice order $\leq_{\mathbb{L}}$ restricted to $J^\infty (\mathbb{L}) \times M^\infty(\mathbb{L})$.
\end{definition}

\begin{definition}
(cf.\ \cite[Definitions 2.3 and 2.12]{Gehrke}) A polarity $\mathbb{P} = (A, X, \perp)$ is:
\begin{enumerate}
\item {\em separating} if the following conditions are satisfied:
\begin{enumerate}
\item[(s1)] for all $a, b\in A$, if $a\neq b$ then $\up{a}\neq \up{b}$, and
\item[(s2)] for all $x, y\in Y$, if $x\neq y$ then $\down{x}\neq \down{y}$.
\end{enumerate}
\item {\em reduced} if the following conditions are satisfied:
\begin{enumerate}
\item[(r1)] for every $a\in A$, some $x\in X$ exists s.t.\ $a$ is $\leq$-minimal in $\{b\in A\mid b \not\perp x \}$.
\item[(r2)] for every $x\in X$, some $a\in A$ exists s.t.\ $x$ is $\leq$-maximal in $\{y\in X\mid x\not \perp a\}$.
\end{enumerate}
\item an {\em RS-polarity}\footnote{In \cite{Gehrke}, RS-polarities are referred to as RS-frames. Here we reserve the term RS-frame for RS-polarities endowed with extra relations used to interpret the operations of the lattice expansion. } if it is separating and reduced.
\end{enumerate}
\end{definition}
If $\mathbb{P}$ is separating, then, denoting $S: = \{b\mid b\in A$ and $b<a\} = a{\downarrow} \setminus \{ a \}$ for each $a\in A$, notice that  $a{\downarrow}$ is completely join-irreducible in $\mathbb{P}^+$ iff $\bigvee_{b\in S}b{\downarrow}\varsubsetneq a{\downarrow}$ iff $\up{a}\varsubsetneq\bigcap_{b\in S}\up{b}$, i.e.\ some $x\in X$ exists such that  $b\perp x$ for all $b\in S$ and $a\not\perp x$, which is condition (r1). Similarly, (r2) dually characterizes the condition that, for every $x\in X$, the subset $x{\uparrow}$ is completely meet-irreducible in $\mathbb{P}^+$, represented as a sub meet-semilattice of $\P(X)$.
\begin{proposition}
(cf.\ \cite[Remark 2.13]{Gehrke} and \cite[Proposition 4.7, Corollary 4.9]{DGP05})
For every perfect lattice $\mathbb{L}$ and RS-polarity  $\mathbb{P}$,
    \begin{enumerate}
    \item $\mathbb{L}_+$ is an RS-polarity and $(\mathbb{L}_+)^{+} \cong \mathbb{L}$.
    \item $\mathbb{P}^+$ is a perfect lattice and $(\mathbb{P}^+)_{+} \cong \mathbb{P}$.
    \end{enumerate}
\end{proposition}

\subsection{RS-frames and models}
\label{ssec:RS-frames and models}
In the present section, we  report on the definition of a relational semantics, based on RS-polarities, for an expansion $\mathcal{L}$  of the basic lattice language  with a unary normal box-type connective. We  also give semantics for a further expansion of $\mathcal{L}$ with a unary normal diamond-type connective $\Diamondblack$, and with two special sorts of variables $\nomi,\nomj$ called {\em nominals}, and $\cnomm, \cnomn$ called {\em co-nominals}. This semantics is the outcome of a dual characterization  which is discussed in detail and in full generality in  \cite[Section 2]{ConPal13}, and is reported on in the appendix for the part directly relevant to this paper. The most peculiar feature of this semantics is that formulas are {\em satisfied} at $a\in A$ and {\em co-satisfied} ({\em refuted}) at   $x\in X$. %naturally interpreted on perfect lattices endowed with a completely meet-preserving operation.

\begin{definition}\label{Def:RSFrame}
An \textit{RS-frame} for $\mathcal{L}$ is a structure $\mathbb{F} = (\mathbb{P},R)$ where $\mathbb{P} = (A,X,\perp)$ is an RS-polarity, and $R\subseteq A\times X$ such that  the images and pre-images of singletons under $R$ are Galois-closed, i.e.\ for every $x\in X$ and $a\in A$,
\begin{center}
$\ud{R^{-1}[x]}\subseteq R^{-1}[x]$ and  $\du{R[a]}\subseteq R[a]$.
\end{center} %is a tuple of additional relations (encoding unary and binary modal operations, respectively) such that $R_{\Diamond}\subseteq Y\times X$,  $R_{\Box}\subseteq X\times Y$, $R_{\lhd}\subseteq Y\times Y$, $R_{\rhd}\subseteq X\times X$, $R_{\circ}\subseteq Y\times X\times X$, and $R_{\star}\subseteq X\times Y\times Y$,  satisfying additional compatibility conditions guaranteeing that the operations associated with the relations in $\mathcal{R}$ map Galois-stable sets to Galois-stable sets.
Relations $R$ which satisfy this condition are called {\em RS-compatible}.
\end{definition}
The additional conditions on $R$ are compatibility conditions guaranteeing that the following assignments respectively define the  operations  $\Box$  and  $\Diamondblack$  associated with $R$  on the lattice $\mathbb{P}^+$: for every $U\in \mathbb{P}^+$,
\begin{center}
$\Box U: = \bigcap \{R^{-1}[x]\mid U\subseteq \down{x}\}\ $ and $\ \Diamondblack U: = \bigvee \{R[a]\mid \ud{a}\subseteq U\}.$
\end{center}
\begin{definition}\label{Def:complex algebra of RS frame}
For every RS-frame $\mathbb{F} = (\mathbb{P},R)$, its \textit{complex algebra} is the lattice expansion  $\mathbb{F}^+: = (\mathbb{P}^+, \Box)$ where $\Box$ is defined as above.
\end{definition}

\begin{lemma}\label{lemma: R circ leq subseteq R}
$\leq\ \circ\ R\ \circ \leq \ \subseteq \ R$ for every RS-frame $\mathbb{F} = (\mathbb{P},R)$.
\end{lemma}
%Willem: Moved to the appendix
%\begin{proof}
%Assume that $aRz$ and  $z\leq y$. To show that $y\in R[a]$, by the second compatibility condition, it is enough to show that $y\in \du{R[a]}$. That is, let us fix $b\in \down{R[a]}$ and show that $b\perp y$. From  $b\in \down{R[a]}$ and $aRz$ it follows that $b\perp z$. This and $z\leq y$ imply that $b\perp y$, given that $\perp\ \circ\ \leq\ \subseteq\ \perp$. The remaining part is proven similarly.
%\end{proof}
%
An \textit{RS-model} for $\mathcal{L}$ on $\mathbb{F}$ is a structure $\mathbb{M} = (\mathbb{F}, v)$ such that $\mathbb{F}$ is an RS-frame for $\mathcal{L}$ and $v$ is a variable assignment mapping each $p\in \mathsf{PROP}$ to a pair $(V_1(p), V_2(p))$ of Galois-stable sets in $\P(A)$ and $\P(X)$ respectively. In a model for the expanded language with $\Diamondblack$, nominals and conominals, variable assignments also map nominals $\nomj$ to  $(\ud{j}, \up{j})$ for some $j$ in $A$ and co-nominals $\cnomm$ to $(\down{m}, \du{m})$ for some $m$ in $X$.

The following table reports the recursive definition of the satisfaction and co-satisfaction relations on $\mathbb{M}$:

\begin{center}
{\footnotesize
\begin{tabular}{l}
\begin{tabular}{lllllll}
$\mathbb{M}, a \Vdash 0$ && never & &$\mathbb{M}, x \succ 0$ && always\\
$\mathbb{M}, a \Vdash 1$ &&always & &$\mathbb{M}, x \succ 1$ &&never\\
$\mathbb{M}, a \Vdash p$ & iff & $a\in V_1(p)$ & &$\mathbb{M}, x \succ p$ & iff & $x\in V_2(p)$\\
$\mathbb{M}, a \Vdash \nomi$ & iff & $a\in V_1(\nomi)$ & &$\mathbb{M}, x \succ \nomi$ & iff & $x\in V_2(\nomi)$\\
$\mathbb{M}, a \Vdash \cnomm$ & iff & $a\in V_1(\cnomm)$ & &$\mathbb{M}, x \succ \cnomm$ & iff & $x\in V_2(\cnomm)$\\
\end{tabular}
\\
\begin{tabular}{llll}
$\mathbb{M}, a \Vdash \phi\wedge \psi$ & iff & $\mathbb{M}, a \Vdash \phi$ and $\mathbb{M}, a \Vdash  \psi$ & \\
$\mathbb{M}, x \succ \phi\wedge \psi$ & iff & for all $a\in A$, if $\mathbb{M}, a \Vdash \phi\wedge \psi$, then $a \perp x$\\
$\mathbb{M}, a \Vdash \phi\vee \psi$ & iff & for all $x\in X$, if $\mathbb{M}, x \succ \phi\vee \psi$, then $a \perp x$  & \\
$\mathbb{M}, x \succ \phi\vee \psi$ & iff &  $\mathbb{M}, x \succ \phi$ and $\mathbb{M}, x \succ  \psi$ &\\
$\mathbb{M}, a \Vdash \Box\phi$ & iff & for all $x\in X$, if $\mathbb{M}, x \succ \phi$, then $a R x$& \\
$\mathbb{M}, x \succ \Box\phi$ & iff & for all $a\in A$, if $\mathbb{M}, a \Vdash \Box\phi$, then $a \perp x$\\
$\mathbb{M}, a \Vdash \Diamondblack\phi$ & iff & for all $x\in X$, if $\mathbb{M}, x \succ \Diamondblack\phi$, then $a \perp x$   &\\
$\mathbb{M}, x \succ \Diamondblack\phi$ & iff &  for all $a\in A$, if $\mathbb{M}, a \Vdash \phi$, then $a R x$. \\
\end{tabular}
\end{tabular}
}
\end{center}
%
%\marginnote{W: I added the following lemma, definition, and remark, to help address reviewer 2's about the equivalences in the examples in Section 2.9.} 
The following lemma is proven easily by simultaneous induction on $\phi$ and $\psi$ using the truth definitions above. The base cases for $0$ and $1$ use conditions (r1) and (r2) and those for proposition letters, nominals and co-nominals follow from the way valuations are defined.

\begin{lemma}\label{lemma:Vdash:Vs:Succ}
For all formulas $\phi$ and $\psi$ it holds that
\begin{enumerate}
\item $\mathbb{M}, a \Vdash \phi$ iff for all $x \in X$, if $\mathbb{M}, x \succ \phi$ then $a \perp x$, and
\item $\mathbb{M}, x \succ \psi$ iff for all $a \in A$, if $\mathbb{M}, a \Vdash \psi$ then $a \perp x$.
\end{enumerate}
\end{lemma}
An inequality $\phi \leq \psi$ is \emph{true} in $\mathbb{M}$, denoted $\mathbb{M} \Vdash \phi \leq \psi$, if for all $a \in A$ and all $x \in X$, if $\mathbb{M}, a \Vdash \phi$ and $\mathbb{M}, x \succ \psi$ then $a \perp x$.

\begin{remark}\label{remark:equiv:version:truth}
It follows from Lemma \ref{lemma:Vdash:Vs:Succ} that $\mathbb{M} \Vdash \phi \leq \psi$ iff for all $a \in A$, if $\mathbb{M}, a \Vdash \phi$ then $\mathbb{M}, a \Vdash \psi$. It also follows that $\mathbb{M} \Vdash \phi \leq \psi$ iff for all $x \in X$, if $\mathbb{M}, x \succ \psi$ then $\mathbb{M}, x \succ \phi$. We will find these equivalent characterizations of truth in a model useful when treating examples.
\end{remark}

\subsection{Standard translation on RS-frames}

As in the Boolean case, each RS-model $\mathbb{M}$ for $\mathcal{L}$ can be seen as a  first-order structure, albeit two-sorted. Accordingly, we define correspondence languages as follows.

Let $L_1$ be the two-sorted first-order language with equality  built over the denumerable and disjoint sets of individual variables $\mathsf{A}$ and $\mathsf{X}$, with binary relation symbol $\perp$,  $R$,  and two unary predicate symbols $P_1, P_2$ for each  $p \in \mathsf{PROP}$.\footnote{The intended interpretation links $P_1$ and $P_2$ in the way suggested by the definition of $\mathcal{L}$-valuations. Indeed, every $p \in \mathsf{PROP}$ is mapped to a pair $(V_1(p), V_2(p))$ of Galois-stable sets as indicated in Subsection \ref{ssec:RS-frames and models}. Accordingly, the interpretation of pairs $(P_1, P_2)$ of predicate symbols is restricted to such pairs of Galois-stable sets, and hence the interpretation of universal second-order quantification is also restricted to range over such sets.}

We will further assume that $L_1$ contains denumerably many individual variables $i, j, \ldots$ corresponding to the nominals $\nomi, \nomj, \ldots \in \mathsf{NOM}$ and $n, m, \ldots$ corresponding to the co-nominals $\cnomn, \cnomm \in \mathsf{CO\text{-}NOM}$. Let $L_0$ be the sub-language which does not contain the unary predicate symbols corresponding to the propositional variables.
Let us now define the \emph{standard translation} of $\mathcal{L}^+$ into $L_1$ recursively:\footnote{Recall that $a\leq j$ abbreviates $\forall x(j\perp x\rightarrow a\perp x)$ and $m \leq x$ abbreviates $\forall a(a\perp m\rightarrow a\perp x)$.}
\begin{center}
{\footnotesize
\begin{tabular}{ll}
$\mathrm{ST}_a(0) := a \neq a$ & $\mathrm{ST}_x(0) := x  = x$ \\
$\mathrm{ST}_a(1) := a = a$ & $\mathrm{ST}_x(1) := x \neq x$\\
$\mathrm{ST}_a(p) := P_1(a)$& $\mathrm{ST}_x(p) :=  P_2(x)$\\
%$\mathrm{ST}_x(p) := P(x)$& $\mathrm{ST}_y(p) := \forall x[ P(x)\rightarrow x\leq y]$\\
$\mathrm{ST}_a(\nomj) := a\leq j$& $\mathrm{ST}_x(\nomj) := j \perp x$  \\
$\mathrm{ST}_a(\cnomm) := a \perp m$ & $\mathrm{ST}_x(\cnomm) := m \leq x$\\
$\mathrm{ST}_a(\phi \vee \psi) := \forall x[\mathrm{ST}_x(\phi\vee \psi)\rightarrow a\perp x]$& $\mathrm{ST}_x(\phi \vee \psi) := \mathrm{ST}_x(\phi) \wedge \mathrm{ST}_x(\psi)$\\
$\mathrm{ST}_a(\phi \wedge \psi) := \mathrm{ST}_a(\phi) \wedge \mathrm{ST}_a(\psi)$& $\mathrm{ST}_x(\phi \wedge \psi) := \forall a[\mathrm{ST}_a(\phi\wedge \psi)\rightarrow a\perp x]$\\
 $\mathrm{ST}_a(\Box \phi) := \forall x [\mathrm{ST}_x(\phi)  \rightarrow aR x]$ & $\mathrm{ST}_x(\Box \phi) := \forall a [\mathrm{ST}_a(\Box \phi)\rightarrow a\perp x]$ \\
$\mathrm{ST}_a(\Diamondblack \phi) := \forall x [\mathrm{ST}_{x}(\Diamondblack \phi)\rightarrow a \perp x]$& $\mathrm{ST}_x(\Diamondblack \phi) := \forall a[\mathrm{ST}_a(\phi)\rightarrow aR x]$\\
\end{tabular}
}
\end{center}
The following is a variant of \cite[Lemma 2.5]{ConPal13}.
\begin{lemma}\label{lemma:Standard:Translation}
For any $\mathcal{L}$-model $\mathbb{M}$  and any $\mathcal{L}^+$-inequality $\phi\leq\psi$, it holds that
%
%\begin{center}
%\begin{tabular}{rcl}
$\mathbb{M} \Vdash \phi \leq \psi$ \quad iff \quad $\mathbb{M} \models \forall a \forall x [\mathrm{ST}_a(\phi) \wedge \mathrm{ST}_x(\psi) \rightarrow  a \perp x]$
\quad iff \quad $\mathbb{M} \models \forall a [\mathrm{ST}_a(\phi) \rightarrow  \mathrm{ST}_a(\psi)]$
\quad iff \quad $\mathbb{M} \models \forall x [\mathrm{ST}_x(\psi) \rightarrow  \mathrm{ST}_x(\phi)]$.
%\end{tabular}
%\marginnote{W: modified this lemma. Please check.}
%\end{center}
\end{lemma}

\subsection{Examples}\label{ssec:examples}
 So far we have seen that the environment of RS-frames provides a mathematically motivated generalization of the correspondence theory which was key to the success of classical normal modal logic as a formal framework in multiple settings. The focus of this paper is to try and understand whether and how this generalized environment can retain some of the intuition which made Kripke semantics and modal logic so appealing. Let us start with the inequality $\Box 0\leq 0$, which corresponds on Kripke frames to the condition that every state has a successor.
\begin{center}
\begin{tabular}{r c l}
$ \Box 0 \leq 0$
&iff &$ \forall a [\mathrm{ST}_{a}(\Box 0) \rightarrow\forall x (\mathrm{ST}_{x}(0) \rightarrow a \perp x)]$\\
&iff &$ \forall a [\forall y(y = y \rightarrow a R y) \rightarrow\forall x (x = x \rightarrow a \perp x)]$\\
&iff &$ \forall a [\forall y(a R y) \rightarrow\forall x (a \perp x)]$\\
&iff & $\forall a\exists y (\neg(aRy))$.\\
\end{tabular}
\end{center}
To justify the last equivalence, notice that by definition, in RS-polarities no object $a$ verifies $\forall x (a \perp x)$. Hence the condition in the penultimate line is true precisely when  the premise of the implication is false. This condition says  that every state is not $R$-related to some co-state; the  condition on Kripke frames is recognizable modulo suitable insertion of  negations. %retains some flavor of the first-order condition in the classical setting.
Next, let us consider the inequality $\Box p\leq p$, which corresponds on Kripke frames to the condition that $R$ is reflexive.
\begin{center}
\begin{tabular}{r c l ll }
$ \forall p (\Box p \leq p)$
&iff &$ \forall \cnomm (\Box \cnomm \leq \cnomm)$\\
&iff &$ \forall a \forall \cnomm[\mathrm{ST}_{a}(\Box \cnomm) \rightarrow\mathrm{ST}_{a}(\cnomm)]$\\
&iff &$ \forall a \forall m(a R m \rightarrow a \perp m)$, & \\
\end{tabular}
\end{center}
 since by definition, $\mathrm{ST}_{a}(\cnomm) = a \perp m$, and  $\mathrm{ST}_{a}(\Box \cnomm)= \forall y(m\leq  y \rightarrow a R y)$ can be rewritten as  $m{\uparrow}\subseteq R[a]$, which is equivalent to $a R m$, since $R\ \circ \leq \ \subseteq R$ (cf.\ Lemma \ref{lemma: R circ leq subseteq R}).
To recognize the connection with the usual reflexivity condition, observe that $ \forall a \forall m(a R m \rightarrow a \perp m)$ is equivalent to $ R\subseteq \ \perp$, and the reflexivity of a relation $R\subseteq A\times A$ can be written as $\text{Id}\subseteq R$, which is equivalent to $R^c\subseteq {\text{Id}}^c$.

Clearly, $\Box p \leq p$ implies $\Box\Box p \leq \Box p$. Let us consider the converse inequality, which in the classical setting corresponds to transitivity:
\begin{center}
\begin{tabular}{r c l ll}
$ \forall p (\Box p \leq \Box\Box p)$
&iff &$ \forall \cnomm (\Box \cnomm \leq \Box\Box \cnomm)$\\
&iff &$ \forall a \forall \cnomm(\mathrm{ST}_{a}(\Box \cnomm) \rightarrow\mathrm{ST}_{a}(\Box\Box \cnomm) )$\\
%iff & $ \forall a \forall \cnomm(a Rm \rightarrow \mathrm{ST}_{a}(\Box\Box \cnomm) )$ & ($\ast$)\\
&iff & $ \forall a \forall m(a Rm \rightarrow R^{-1}[m]^{\uparrow}\subseteq R[a]),$\\
\end{tabular}
\end{center}
where
\begin{center}
\begin{tabular}{c l ll}
$\mathrm{ST}_{a}(\Box\Box \cnomm)$& $  = $ & $\forall y[\mathrm{ST}_{y}(\Box \cnomm) \rightarrow a R y]$\\
&$ = $ & $\forall y[\forall b(\mathrm{ST}_{b}(\Box \cnomm)\rightarrow b\perp y) \rightarrow a R y]$\\
&$ = $ & $\forall y[\forall b(bRm \rightarrow b\perp y) \rightarrow a R y]$ & ($\ast\ast$)\\
&$ = $ & $R^{-1}[m]^{\uparrow}\subseteq R[a]$.
\end{tabular}
\end{center}

While, again, with a bit of work it is possible to retrieve the transitivity condition in this new interpretation, already with a relatively simple inequality such as $\Box p \leq \Box\Box p$ this game  is not really useful for the purpose of gaining a better intuitive understanding of this semantics, since it requires jumping through too many hoops (the accessibility relation on states is here encoded into a `non unaccessibility' relation between states and co-states), and quickly becomes awkward and unintuitive. In the next section, we will argue that better results can be achieved by taking it as primitive, rather than as the generalization of some other semantics. % are going to provide a `positive' interpretation of this semantics, i.e.\ .

\section{Conceptualizing RS-semantics via categorization theory}
\label{sec:categorization}
In the present section, we  propose a conceptualization of the notions introduced in the previous section based on ideas from  categorization theory in management science.
The starting point of this conceptualization is the very well known idea, core to Formal Concept Analysis, that polarities $(A, X, \perp)$ are  abstract representations of databases, in which $A$ and $X$ are sets of {\em objects} and {\em properties} respectively, and $\perp$ encodes information about whether a given object satisfies a given property.
More specifically, we propose to think of a given polarity $\mathbb{P} = (A, X, \perp)$ as a database such that $A$ is the set of all  {\em products} in a given market at a certain moment (e.g.\ all models of cars, or models of togas on sale in the Netherlands in a given year),  and $X$ all the relevant observable {\em features} of these products. %(where we understand that a feature is relevant if it has been taken in consideration in the decision making by some agent).
The specialization pre-order $a\leq b$ on objects ($a$ has at least all the features that $b$ has)  can then be read as `product $a$ is at least as specified (i.e.\ rich in features) as product $b$' and the one on features $x\leq y$ (any  product having $x$ has also $y$) as `feature $y$ is more generic than feature $x$'.
The RS-conditions on the database can then be understood as follows:
\begin{itemize}
\item[](s1): Any two distinct products can be told apart by some  feature;

\item[](s2): For any two distinct features there is a product having one but not the other;

\item[](r1): For any product $a$, if there are strictly more specified products than $a$ in the market, then they all share some feature $x$ which $a$ does not have;

\item[](r2): For any feature $x$, if there are strictly more generic features than $x$, then some product $a$ exists which has all of them but not $x$.
\end{itemize}
The separation conditions (s1) and (s2) seem rather intuitive and do not require much explanation; (r1) can be enforced by suitably adding `artificial' features to the database, and (r2) can be enforced by  removing features from the database which are the exact intersection of two or more generic features.\footnote{For instance, consider the following features of a soft drink: $x$: = `with vitamin A', $y$: = `with vitamin C', $z$: = with vitamin A and C'. Clearly, a database with these features would violate (r2). This can be remedied by removing $z$ from the set $X$ of  the database.} Clearly, removing such features can always be done without loss of descriptive power. We can always enforce the separation and reduction conditions, since the finite polarities we consider are a subclass of the so-called doubly founded polarities, for which this is always possible, see \cite{ganter1997applied}. 
%\marginnote{it might be the case that we can always enforce r1 and r1 separately but not simultaneously. That is, the feature that we add to enforce r1 create a problem w.r.t.\ r2. Shall we address this issue?} %understood as a saturation condition, which can be observed in markets which are old and established enough (for instance pop music) that  enough differentiation among products has taken place; \marginred{original laptop did not have wifi, but now all  do}
%
%finally, a limit case in which (r2) holds is in databases such that the features are mutually exclusive (i.e.\ `co-atomic'); for instance, consider a large enough music database in which features are e.g.\ the number of beats per minute, or the duration of  tracks.

%Independently of how plausible or widespread these conditions turn out to be in real markets,

Arguably, the reformulation of the RS-conditions in terms of products and features makes them easier to grasp. %, and this is all the present paper aims to achieve.

Further, we propose to understand the lattice $\mathbb{P}^+$ as the collection of `candidate categories'. That is, each element of $\mathbb{P}^+$ is a set of products which is completely identified by the set of features  common to its elements. That is,  any product with all these features is a member of the `candidate category'. We refer to these categories as `candidate' since they are purely implicit in the database, and not necessarily the target of any  social construction. In particular, only a restricted subset of candidate categories will support the interpretation of socially meaningful categories (which have {\em labels} such as western, opera,  bossa-nova, SUV, smart phone etc.).
Labels of socially meaningful categories can be assigned to `candidate categories' in the usual way, namely, by means of an assignment $v$  which associates each {\em atomic category label} $p\in \mathsf{PROP}$ to a category viewed both extensionally as $V_1(p)\subseteq X$ and intensionally as $V_2(p)\subseteq A$.\footnote{ Recall that for such an assignment, $V_1(p) = \down{V_2(p)}$ and  $V_2(p) = \up{V_1(p)}$.} Notice the perfect match between the encoding of the {\em meaning} of atomic propositions  on Kripke models and  of atomic category labels on RS-models: the meaning of  atomic proposition $p$ is given as the set of states at which $p$ holds true;  the meaning of atomic category  label $p$ is given as the set of products which are the {\em members} of $p$, and the set of features which {\em describe} $p$. In what follows, we will refer to the intension of a category (cf.\ Footnote \ref{footnote:intension}) as its {\em description}, and we say that a feature {\em describes} a category if it belongs to its description.

Given such an assignment,\footnote{Empirically, there are many ways to generate such an assignment \cite{paleo2006classification}.} the database is endowed with a structure of an $\mathcal{L}$-model $\mathbb{M}$, in such a way that, for every formula (category label) $\phi\in \mathcal{L}$, any $a\in A$ and $x\in X$, the symbols $\mathbb{M}, a \Vdash \phi$ and $\mathbb{M}, x \succ \phi$ can be understood as `object $a$ is a member of category $\phi$', and `feature $x$ describes category $\phi$'. One immediately apparent advantage of this conceptualization is that it provides an intuitive way to understand $\succ$ from first principles rather than  as the negative counterpart of $\Vdash$.

The other advantage concerns the understanding of the connectives $\wedge$ and $\vee$ in the general lattice environment. The issue is that their standard interpretation as conjunction and disjunction does not seem completely right, since distributivity seems hardwired in the way we understand `and' and `or' in natural language.
The satisfaction clauses for $\wedge$ and $\vee$ formulas read:
\begin{center}
\begin{tabular}{llll}
$\mathbb{M}, a \Vdash \phi\wedge \psi$ & iff & $\mathbb{M}, a \Vdash \phi$ and $\mathbb{M}, a \Vdash  \psi$ & \\
$\mathbb{M}, x \succ \phi\wedge \psi$ & iff & for all $a\in A$, if $\mathbb{M}, a \Vdash \phi\wedge \psi$, then $a \perp x$\\
$\mathbb{M}, a \Vdash \phi\vee \psi$ & iff & for all $x\in X$, if $\mathbb{M}, x \succ \phi\vee \psi$, then $a \perp x$  & \\
$\mathbb{M}, x \succ \phi\vee \psi$ & iff &  $\mathbb{M}, x \succ \phi$ and $\mathbb{M}, x \succ  \psi$ &\\
\end{tabular}
\end{center}
These clauses say that the category $\phi\wedge \psi$ is the one whose members are  members of both categories $\phi$ and $\psi$; hence, these products will satisfy at least both the description of $\phi$ and of $\psi$, and hence the description of $\phi\wedge \psi$ contains at least the union of these descriptions.  The category $\phi\vee \psi$ is described by the intersection of the descriptions of $\phi$ and of $\psi$. Hence, membership in $\phi\vee \psi$ only requires products to satisfy this smaller set of features, and  typically includes much more than the union of the members of the two categories. So for instance, $\mathsf{bird}\vee\mathsf{cat}$ would exclude reptiles, insects and fish, but include vertebrate homeothermic species such as the platypus. This interpretation of $\wedge$ and $\vee$ makes it possible to understand intuitively why distributivity fails. Indeed,  a member of $(\mathsf{phone}\vee \mathsf{smartphone}) \wedge (\mathsf{kettle}\vee \mathsf{smartphone})$ is guaranteed to have all the features in the description of $\mathsf{phone}$ (and in fact, $\mathsf{kettle}\vee \mathsf{smartphone}$ is so general that can be assumed to not add any feature that $\mathsf{phone}$ does not have already). However, this might be not enough for it to be a member  of  $(\mathsf{phone}\wedge \mathsf{kettle})\vee \mathsf{smartphone}$, given that the category $\mathsf{phone}\wedge \mathsf{kettle}$ has no members (hence its description consists of {\em all}  features), and so the members of $(\mathsf{phone}\wedge \mathsf{kettle})\vee \mathsf{smartphone}$ must have at least all the features in the description of $\mathsf{smartphone}$.

Now that we have a working understanding of $\Vdash$ and $\succ$, we can recognize the normal box-type operator on $\mathbb{P}^+$ as the perspective of a single  agent on  categories. Accordingly, $\mathbb{M}, a \Vdash \Box\phi$ and $\mathbb{M}, x \succ \Box\phi$ can be understood as `object $a$ is a member of category $\phi$ according to the agent', and `feature $x$ describes category $\phi$ according to the agent'. The normality conditions $\Box \top = \top$ and $\Box (\phi\wedge \psi) = \Box\phi\wedge \Box\psi$ can be understood as rationality requirements: that is, the agent correctly recognizes the `uninformative' category $\top$ as such, and her understanding/perception of the greatest common subcategory of any two categories $\phi$ and $\psi$ is the greatest common subcategory of the categories she understands as $\phi$ and $\psi$.

On the side of the database, the agent is modelled as a relation $R\subseteq A\times X$. Hence, $aRx$ intuitively reads `object $a$ has feature $x$ according to the agent'. Unsurprisingly, the additional  properties of $R$ (cf.\ Lemma \ref{lemma: R circ leq subseteq R}) can be also understood as rationality requirements: if $aRx$ then $aRy$ for every $y\geq x$ says that if the agent attributes feature $x$ to product $a$, then the agent will attribute to $a$ also all the features which are `implied' by $x$. Likewise,  if $aRx$ then $bRx$ for every $b\leq a$ says that if the agent attributes feature $x$ to product $a$, then the agent will attribute $x$ also to all the products which are `more specified' than $a$. %This relation gives rise to a box-type normal operator $\Box$, which encodes the epistemic stance of the agent towards the latent categories. The categories in the image of $\Box$ are the categories perceived by the agent.

Like in the classical case, two modal operators, $\Box$ and $\Diamondblack$, are associated with the same relation $R$. However, these operations are not dual to each other, in the sense of e.g.\  $\Diamond: = \neg\Box\neg$, but are rather adjoints to each other, that is, for all $u, v\in \mathbb{P}^+$,
\begin{center}
$\Diamondblack u\leq v$ \ iff \ $u\leq \Box v$.
\end{center}
In fact, rather than encoding the dual perspective on the subjectivity of the agent that  $\Box$ encodes, the operation $\Diamondblack$ encodes the same perspective that $\Box$ encodes, only geared towards objects while $\Box$ is geared towards features.
Indeed, for every object $j$ and every feature $m$, denoting by $\nomj$ and $\cnomm$ the categories respectively generated by  $j$ and  $m$,
\begin{center}
$\Diamondblack \nomj\leq \cnomm$ \ iff \  $jRm$ \ iff \ $\nomj\leq \Box \cnomm$.
\end{center}
Thus, the information $jRm$ (`the agent attributes feature $m$ to object $j$) is encoded  on the side of the categories both by saying that $m$ describes the category $\Diamondblack \nomj$ (the one the  agent understands as the category generated by  $j$), and by saying that $j$ is a member of the category $\Box\cnomm$ (the one the  agent understands as the category generated by  $m$).
As to the defining clauses of the recursive definition of $\Vdash$ and $\succ$,
 by definition, $\mathbb{M}, a\Vdash \Box \phi$ is the case iff  for all features $x$, if $\mathbb{M}, x \succ \phi$, then $a R x$. That is,  product $a$ is  recognized by the agent as member of  category $\phi$ iff the agent attributes to $a$ all the features that belong to the description  of $\phi$.

Moreover, by definition, $\mathbb{M}, x \succ \Box\phi$  iff  for all $a\in A$, if $\mathbb{M}, a \Vdash \Box\phi$, then $a \perp x$. That is,  feature $x$ pertains to the description  of  category $\phi$ according to the agent iff $x$ is verified by each object $a$ that the agent recognizes as  a member of $\phi$.
 %\marginred{shall we discuss also the interpretation of $\wedge$ and $\vee$?}

% $\mathbb{M}, a \Vdash \Diamondblack\phi$ & iff & for all $x\in X$, if $\mathbb{M}, x \succ \Diamondblack\phi$, then $a \perp x$   &\\
%$\mathbb{M}, x \succ \Diamondblack\phi$ & iff &  for all $a\in A$, if $\mathbb{M}, a \Vdash \phi$, then $a R x$. \\

Two modal axioms commonly considered in epistemic logic are `reflexivity' $\Box p\leq p$ and `transitivity' ($\Box p \leq \Box\Box p$). The axiom $\Box p\leq p$ is interpreted  epistemically as the {\em factivity} of knowledge (`if the agent knows that $p$ then $p$ is true'). The first-order correspondent of the factivity axiom on RS-frames is $\forall a \forall x(aRx\rightarrow a\perp x)$, which indeed expresses a form of factivity, in that it requires that whenever the agent attributes any feature $x$ to any product $a$, then it is indeed the case that $x$ is a feature of $a$.
The axiom $\Box p\leq \Box\Box  p$ is interpreted  epistemically as the {\em positive introspection} of knowledge (`if the agent knows that $p$, then the agent knows that she knows that $p$'). The first-order correspondent of the positive introspection axiom on RS-frames is  $\forall a \forall m(a Rm \rightarrow R^{-1}[m]^{\uparrow}\subseteq R[a])$, expressing the condition that if an agent attributes feature $m$ to product $a$, then she will attribute to $a$ all the features which are shared by the products to which she attributes $m$. %This can be interpreted as a rationality requirement, since by assumption $a$ is among the products to which the agent attributes $m$.\marginred{can we make the link with positive introspection more precise?}
To understand the link between this condition and positive introspection, consider the category $\Box \cnomm$, i.e.\ the category which the agent understands as the one generated by a given feature $m$.\footnote{In fact, the same argument would hold more in general for any category $\Box\phi$.} This category can be identified with the tuple $(R^{-1}[m], R^{-1}[m]^{\uparrow})$. That is, the members of $\Box \cnomm$ are the products to which the agent attributes $m$ (recall that $R^{-1}[m]$ is a Galois-stable set by Definition \ref{Def:RSFrame}) and the description of $\Box \cnomm$ is the set of the features which the products in $R^{-1}[m]$ have in common. By definition,  $b\perp z$ for every $b\in R^{-1}[m]$ and $z\in R^{-1}[m]^{\uparrow}$. The first-order correspondent of $\Box p \leq \Box\Box p$ requires that  $b R z$ for such $b$ and $z$.   So, while factivity corresponds to $R\subseteq \ \perp$, positive introspection gives the reverse inclusion restricted to products  and features pertaining to `boxed categories'. That is, the agent must be aware of the features of the products  of the categories that she knows.

\section{Categories as social constructs}
\label{sec:social-construction}
In the present section, we introduce a formal account of the emergence of categories as the outcome of a process of social interaction. We consider for the sake of simplicity a setting of two agents. Accordingly, we consider the bi-modal logic $\mathcal{L}$ which is the axiomatic extension of the basic normal LE-logic for two unary normal box-type modal operators, 1 and 2, with the axioms $ip\leq p$ and $ip\leq iip$ for $1\leq i\leq 2$. Models for this logic are structures $(\mathbb{P}, R_1, R_2, v)$ such that $\mathbb{P} = (X, A, \perp)$ is an RS-polarity, $R_i\subseteq A\times X$  for $1\leq i\leq 2$, such that the following conditions hold:
\begin{enumerate}
\item $\forall x(\ud{R_i^{-1}[x]}\subseteq R_i^{-1}[x])$;
\item  $\forall a(\du{R_i[a]}\subseteq R_i[a])$;
\item $ R_i\subseteq \ \perp$;
\item $\forall a \forall x(a R_i x \rightarrow R_i^{-1}[x]^{\uparrow}\subseteq R_i[a])$,
\end{enumerate}
and $v$ is an assignment which associates each $p\in \mathsf{PROP}$ to an element of $\mathbb{P}^+$ viewed both extensionally as $V_1(p)\subseteq A$ and intensionally as $V_2(p)\subseteq X$ in such a way that $V_1(p) = \down{V_2(p)}$ and  $V_2(p) = \up{V_1(p)}$.

In this setting, a common knowledge-type construction can be performed which yields an expansion, denoted  $\mathcal{L}_C$, of the bi-modal LE-logic above with a  normal box-type operator $C$, the interpretation of which on $\mathbb{P}^+$, given the additional axioms, is given as follows: for any $u\in \mathbb{P}^+$,
\begin{center}
$C(u): = \bigwedge_{s\in S}su,$
\end{center}
where $S$ is the set of all compound modalities of the forms  $(ij)^n$ and $i(ji)^n$,  for $1\leq i\neq j\leq 2$ and for some $n\in \mathbb{N}$.
\begin{lemma}
\label{lemma:properties of C}
 $C(u)\leq u$ and $C(u)\leq C(C(u))$ for any $u\in \mathbb{P}^+$.
\end{lemma}
%Willem: Moved proof to the appendix
%\begin{proof}
%Clearly, $C(u)\leq 1u\leq u$, which proves the first inequality.
%\begin{center}
%$C(C(u)) = \bigwedge_{s\in S}sC(u) = \bigwedge_{s\in S}s(\bigwedge_{t\in S}tu) = \bigwedge_{s\in S}\bigwedge_{t\in S}stu \geq \bigwedge_{s'\in S}s'u = C(u)$.
%\end{center}
%\end{proof}
Let $R_C, R_s\subseteq A\times X$ for any $s\in S$ be defined as follows: $aR_s x$ iff $a\leq s x$ \ and \ $aR_C x$ iff $a\leq C(x)$. Clearly, $R_C = \bigcap_{s\in S} R_s$. In the standard setting of epistemic logic, the accessibility relations associated with agents do not directly encode the agents' knowledge but rather their uncertainty. Hence, on the relational side, the relation associated with the common knowledge operator is defined as the reflexive transitive closure of the the union of the relations associated with individual agents, which is typically much bigger than those associated with individual agents. In the present setting, relations associated with agents directly encode what agents  positively {\em know} rather than their uncertainty. Consequently, the common knowledge relation $R_C$ is the intersection of the relations $R_s$ encoding the finite iterations, which is typically much smaller.

As both $C$ and every $s\in S$ are compositions of normal box-operators, they are themselves normal box-operators. Hence  the relations $R_C$ and $R_s$ they give rise to are RS-compatible (cf.\ Definition \ref{Def:RSFrame}). Thus, the correspondence reductions discussed in Section \ref{ssec:examples} apply to $C$ and $R_C$, yielding:
\begin{lemma}
The relation $R_C$ defined above verifies the following conditions:
\begin{enumerate}
\item $ R_C\subseteq \ \perp$;
\item $\forall a \forall x(a R_C x \rightarrow R_C^{-1}[x]^{\uparrow}\subseteq R_C[a])$.
\end{enumerate}
\end{lemma}
For any given category label  $\phi$, the category $C(\phi) = \bigwedge\{C(\cnomm)\mid \phi^{\mathbb{P}^+}\leq \cnomm\}$. For this reason, in what follows we restrict our attention to categories $C(\cnomm)$ for some feature $m\in X$. %\marginred{we need to speak about valuations for the expanded language somewhere in the preliminaries section}
The members of $C(\cnomm)$ are the products in the set $R^{-1}_C[m] = (\bigcap_{s\in S}R_s)^{-1}[m]$, and the description of $C(\cnomm)$ is $R^{-1}_C[m]^{\uparrow} = ((\bigcap_{s\in S}R_s)^{-1}[m])^{\uparrow}$. These can be understood as the socially constructed categories, the membership and description of which are socially agreed upon. Clearly, there are many less of them than candidate categories, which agrees with our intuition. 
%\subsection{Language}
%\label{ssec:Language}
%\input{Language}

%\subsection{Models}
%\label{ssec:Models}
%\input{Models}

%\subsection{Semantics}
%\label{ssec:Semantics}
%\input{Semantics}

%\section{Examples}

%\input{Examples.tex}

\section{Conclusion and further research}
\label{CCL}

In this paper we have proposed an interpretation of RS-semantics in terms of agents' reasoning about objects, their properties and the categories induced by the accompanying relation. We have argued that this semantics is particularly well adapted to this interpretation and, conversely, that through this interpretation one could gain an intuitive understanding of the semantics.

Our proposal has a distinctly epistemic character, but one which differs from standard epistemic logic in at least two respects: firstly, the relations used to interpret the epistemic operators are intended to capture positive knowledge, rather than uncertainty;  secondly, these relations relate objects to features rather than possible worlds to one another. We considered two classical principles of epistemic logic, namely factivity and positive introspection. By applying the correspondence theory of \cite{ConPal13} we computed the relational properties corresponding to these principles, i.e.\ necessary and sufficient conditions on an agent's incidence relation between objects and properties for her knowledge of categories to verify these epistemic principles. Various questions for further investigation remain open here: what is the meaning of other classical epistemic principles, like e.g.\ negative introspection, in this setting?
%Could the principles we consider be reasonably weakened?
Are there other principles that should be included in a minimal logic of categorization? Of course, all of this depends on the reasoning abilities and level of access to reality we wish to attribute to agents. Moreover, most standard logical questions remain open: axiomatizations, proof systems, decidability, complexity, etc.

This paper is a first assay in using RS-semantics for reasoning about categorization and, as such, remains quite general in its assumptions. To be of more immediate practical relevance, the  considerations here should be specialized to particular fields of enquiry where categorization plays or could play a prominent role. Below we briefly consider three such fields.

{\it Natural language semantics.} We have seen that the assignments of RS-models support a notion of meaning that is different from the one in classical modal logic, but is recognizably what the meaning of category labels should be: namely, a semantic category specified as the set of its members and the set of features describing it. In natural language semantics, linguistic utterances are assigned a meaning in the same spirit, which generalizes the truth-based semantics of sentences. More generally, categories or concepts are fundamental to the construction of meaning in natural language, since each noun  is naturally associated with a category. Exploring systematic connections between categories and natural language semantics is a promising direction for further research.

{\it Knowledge representation and formal ontologies.} Categories are central to any form of knowledge representation. Description logics \cite{descriptionlogic2003handbook} are one of the dominant paradigms for logical reasoning in  this context. Our formalism represents a different and possibly complementary perspective on the formal ontologies, classification systems, and taxonomies studied there. In particular, the non-distributive nature of category formation and
the two-level separation between objects and features are foreign to the description logics paradigm. It is natural to ask to what degree the various expressive features of description logic (like uniqueness quantification, qualified cardinality restrictions etc.) could be accommodated in in our framework, and future extensions will study this question.

%Willem: I'm commenting this out for the sake of the WOLLIC paper
%\paragraph{Syllogistic logic.} subcategory inclusions encode statements of the type ``Every A is a B'' or we can also have ``some As are B'' or the other two.\marginnote{Expand on this.}

{\it Categorization theory in management science.} As already indicated, this was one of our main sources of inspiration for the proposals of the present paper. Our formalism is a first step in the direction of a formal logical account of the real world phenomena studied by categorization theorists. There are various considerations that make it an attractive framework in which to study categorization and in which to formulate empirically testable hypotheses. We mention two of these reasons: Firstly, it allows one to study the effects of adding or removing objects with new properties and/or properties already associated with other categories, thus allowing for a fine grained analysis of the likely changes in a classification system resulting from innovations of different kinds.

Secondly, our approach gives us all potential categories ``automatically", while only some of them are real, socially agreed upon categories for economic decision-makers. It can therefore serve as a powerful instrument to better study and understand the causes and consequences of the selection of real categories from the broader set of potential ones. To start with, different real world domains could be compared with respect to the ratios of real to potential categories present in them. One reasonable conjecture seems to be that these ratios will depend a lot on competitive dynamics and the matureness of categories, while also having an effect on them. One could the go on to study changes over time in these ratios as well as the differences in ratios---and their changes---among different audiences espousing different classifications.

{\it Extensions and variations.} In closing, we mention two of the many possible extensions of the present framework:  category membership does not need to be absolute, as  products can simultaneously have different grades of membership in different categories. This calls for quantitative, possibly many-valued versions of our semantics. Also, the categories in a given market do not need to be static, but can evolve and change  over time as new products with new features or new combinations of existing features enter the market \cite{wijnberg2004innovation,wijnberg2011classification}. Dynamic versions of our formalism would be suitable to deal with such continuously evolving categorization systems.  

%\marginnote{W: The references to the arxiv versions of our papers are missing.}
\bibliographystyle{plain}
\bibliography{aiml16}

\begin{thebibliography}{10}

\bibitem{descriptionlogic2003handbook}
Franz Baader, Diego Calvanese, Deborah~L. McGuinness, Daniele Nardi, and
  Peter~F. Patel{-}Schneider, editors.
\newblock {\em The Description Logic Handbook: Theory, Implementation, and
  Applications}. Cambridge University Press, 2003.

\bibitem{baltag1998logic}
Alexandru Baltag, Lawrence~S Moss, and Slawomir Solecki.
\newblock The logic of public announcements, common knowledge, and private
  suspicions.
\newblock In {\em Proc.\ of the 7th conference on Theoretical aspects of
  rationality and knowledge}, pages 43--56. Morgan Kaufmann Publishers Inc.,
  1998.

\bibitem{bimbo2001four}
Katalin Bimb{\'o}, J~Michael Dunn, et~al.
\newblock Four-valued logic.
\newblock {\em Notre Dame Journal of Formal Logic}, 42(3):171--192, 2001.

\bibitem{CoCr14}
Willem Conradie and Andrew Craig.
\newblock Canonicity results for mu-calculi: an algorithmic approach.
\newblock {\em Journal of Logic and Computation}, 27(3):705--748, 2017. ArXiv
  preprint arXiv:1408.6367.

\bibitem{CoCrPaZh16}
Willem Conradie, Andrew Craig, Alessandra Palmigiano, and Zhiguang Zhao.
\newblock Constructive canonicity for lattice-based fixed point logics.
\newblock volume 10388 of {\em Lecture Notes in Computer Science}, pages
  92--109. Springer, 2017.
\newblock ArXiv preprint arXiv:1603.06547.

\bibitem{CoFoPaSo15}
Willem Conradie, Yves Fomatati, Alessandra Palmigiano, and Sumit Sourabh.
\newblock Algorithmic correspondence for intuitionistic modal mu-calculus.
\newblock {\em Theoretical Computer Science}, 564:30--62, 2015.

\bibitem{CoGhPa14}
Willem Conradie, Silvio Ghilardi, and Alessandra Palmigiano.
\newblock Unified correspondence.
\newblock In Alexandru Baltag and Sonja Smets, editors, {\em Johan van Benthem
  on Logic and Information Dynamics}, volume~5 of {\em Outstanding
  Contributions to Logic}, pages 933--975. Springer International Publishing,
  2014.

\bibitem{ConPal13}
Willem Conradie and Alessandra Palmigiano.
\newblock Algorithmic correspondence and canonicity for non-distributive
  logics.
\newblock {\em Annals of Pure and Applied Logic}.
\newblock Forthcoming, preliminary version on arXiv:1603.08515 [math.LO].

\bibitem{CoPa16}
Willem Conradie and Alessandra Palmigiano.
\newblock Constructive canonicity of inductive inequalities.
\newblock Submitted, preliminary version on arXiv:1603.08341 [math.LO].

\bibitem{ConPal12}
Willem Conradie and Alessandra Palmigiano.
\newblock Algorithmic correspondence and canonicity for distributive modal
  logic.
\newblock {\em Annals of Pure and Applied Logic}, 163(3):338 -- 376, 2012.

\bibitem{ConPalSou}
Willem Conradie, Alessandra Palmigiano, and Sumit Sourabh.
\newblock Algebraic modal correspondence: {S}ahlqvist and beyond.
\newblock {\em Journal of Logical and Algebraic Methods in Programming},
  91:60--84, 2017.

\bibitem{CoPaSoZh15}
Willem Conradie, Alessandra Palmigiano, Sumit Sourabh, and Zhiguang Zhao.
\newblock Canonicity and relativized canonicity via pseudo-correspondence: an
  application of {ALBA}.
\newblock Submitted, preliminary version on arXiv:1511.04271 [cs.LO].

\bibitem{CoPaZh15}
Willem Conradie, Alessandra Palmigiano, and Zhiguang Zhao.
\newblock {\em {L}ogical {M}ethods in {C}omputer {S}cience}.
\newblock To appear, preliminary version on arXiv:1603.08220 [math.LO].

\bibitem{CoRo14}
Willem Conradie and Claudette Robinson.
\newblock On sahlqvist theory for hybrid logics.
\newblock {\em Journal of Logic and Computation}, 27(3):867--900, 2017.

\bibitem{crapo1982unities}
Henry Crapo.
\newblock Unities and negation: on the representation of finite lattices.
\newblock {\em Journal of Pure and Applied Algebra}, 23(2):109--135, 1982.

\bibitem{DaveyPriestley2002}
B.~A. Davey and H.~A. Priestley.
\newblock {\em Lattices and {O}rder}.
\newblock Cambridge {Univerity} {P}ress, 2002.

\bibitem{dunn1990gaggle}
J~Michael Dunn.
\newblock Gaggle theory: an abstraction of galois connections and residuation,
  with applications to negation, implication, and various logical operators.
\newblock In {\em Logics in AI}, pages 31--51. Springer, 1990.

\bibitem{DGP05}
J.~Michael Dunn, Mai Gehrke, and Alessandra Palmigiano.
\newblock Canonical extensions and relational completeness of some
  substructural logics.
\newblock {\em Journal Symbolic Logic}, 70(3):713--740, 2005.

\bibitem{duntsch2004relational}
Ivo D{\"u}ntsch, Ewa Or{\l}owska, A~Radzikowska, and Dimiter Vakarelov.
\newblock Relational representation theorems for some lattice-based structures.
\newblock {\em Journal of Relation Methods in Computer Science}, 1:132--160,
  2004.

\bibitem{fagin2003reasoning}
Ronald Fagin, Yoram Moses, Moshe~Y Vardi, and Joseph~Y Halpern.
\newblock {\em Reasoning about knowledge}.
\newblock MIT press, 2003.

\bibitem{FrPaSa14}
Sabine Frittella, Alessandra Palmigiano, and Luigi Santocanale.
\newblock Dual characterizations for finite lattices via correspondence theory
  for monotone modal logic.
\newblock {\em Journal of Logic and Computation}, 27(3):639--678, 2017.

\bibitem{galatos2007residuated}
Nikolaos Galatos, Peter Jipsen, Tomasz Kowalski, and Hiroakira Ono.
\newblock {\em Residuated Lattices: An Algebraic Glimpse at Substructural
  Logics}.
\newblock Elsevier, 2007.

\bibitem{ganter1997applied}
Bernhard Ganter and Rudolf Wille.
\newblock Applied lattice theory: Formal concept analysis.
\newblock In {\em In General Lattice Theory, G. Gr{\"a}tzer editor,
  Birkh{\"a}user}. Citeseer, 1997.

\bibitem{Wille}
Bernhard Ganter and Rudolf Wille.
\newblock {\em Formal concept analysis: mathematical foundations}.
\newblock Springer Science \& Business Media, 2012.

\bibitem{Gehrke}
Mai Gehrke.
\newblock Generalized kripke frames.
\newblock {\em Studia Logica}, 84(2):241--275, 2006.

\bibitem{GMPTZ}
G.~Greco, M.~Ma, A.~Palmigiano, A.~Tzimoulis, and Z.~Zhao.
\newblock Unified correspondence as a proof-theoretic tool.
\newblock {\em Journal of Logic and Computation}, 28(7):1367--1442, 2018.
\newblock arXiv:1603.08204 [math.LO].

\bibitem{hartonas15modal}
Takis Hartonas.
\newblock Modal and temporal extensions of non-distributive logics.
\newblock 2015.

\bibitem{hartonas15order}
Takis Hartonas.
\newblock Order-dual relational semantics for non-distributive propositional
  logics: A general framework.
\newblock 2015.

\bibitem{hsu2011typecasting}
Greta Hsu, Michael~T Hannan, and L{\'a}szl{\'o} P{\'o}los.
\newblock Typecasting, legitimation, and form emergence: A formal theory.
\newblock {\em Sociological Theory}, 29(2):97--123, 2011.

\bibitem{jarvinen2005relational}
Jouni J{\"a}rvinen and Ewa Orlowska.
\newblock Relational correspondences for lattices with operators.
\newblock In {\em Relational Methods in Computer Science}, pages 134--146.
  Springer, 2005.

\bibitem{kamide2002kripke}
Norihiro Kamide.
\newblock Kripke semantics for modal substructural logics.
\newblock {\em Journal of Logic, Language and Information}, 11(4):453--470,
  2002.

\bibitem{wijnberg-gap}
Bram Kuijken, Mark~A.A.M. Leenders, Nachoem~M. Wijnberg, and Gerda Gemser.
\newblock The producer-consumer classification gap and its effects on music
  festival success.
\newblock 2016.
\newblock submitted.

\bibitem{Kurtonina}
N.~Kurtonina.
\newblock {C}ategorical {I}nference and {M}odal {Lo}gic.
\newblock {\em {J}ournal of {L}ogic, {L}anguage, and {I}nformation}, 7, 1998.

\bibitem{moshier2014}
M~Andrew Moshier and Peter Jipsen.
\newblock Topological duality and lattice expansions, {II}: Lattice expansions
  with quasioperators.
\newblock {\em Algebra universalis}, 71(3):221--234, 2014.

\bibitem{navis2010new}
Chad Navis and Mary~Ann Glynn.
\newblock How new market categories emerge: Temporal dynamics of legitimacy,
  identity, and entrepreneurship in satellite radio, 1990--2005.
\newblock {\em Administrative Science Quarterly}, 55(3):439--471, 2010.

\bibitem{paleo2006classification}
Iv{\'a}n~Orosa Paleo and Nachoem~M Wijnberg.
\newblock Classification of popular music festivals: A typology of festivals
  and an inquiry into their role in the construction of music genres.
\newblock {\em International Journal of Arts Management}, pages 50--61, 2006.

\bibitem{PaSoZh15a}
Alessandra Palmigiano, Sumit Sourabh, and Zhiguang Zhao.
\newblock {J{\'o}nsson-style canonicity for ALBA-inequalities}.
\newblock {\em Journal of Logic and Computation}, 27(3):817--865, 2017.

\bibitem{PaSoZh15b}
Alessandra Palmigiano, Sumit Sourabh, and Zhiguang Zhao.
\newblock Sahlqvist theory for impossible worlds.
\newblock {\em Journal of Logic and Computation}, 27(3):775--816, 2017.

\bibitem{ploscica1995}
Miroslav Plo{\v{s}}cica.
\newblock A natural representation of bounded lattices.
\newblock {\em Tatra Mountains Math. Publ}, 5:75--88, 1995.

\bibitem{Suz11}
T.~Suzuki.
\newblock Canonicity results of substructural and lattice-based logics.
\newblock {\em Review of Symbolic Logic}, 4:1 -- 42, 2011.

\bibitem{wijnberg2004innovation}
Nachoem~M Wijnberg.
\newblock Innovation and organization: Value and competition in selection
  systems.
\newblock {\em Organization studies}, 25(8):1413--1433, 2004.

\bibitem{wijnberg2011classification}
Nachoem~M Wijnberg.
\newblock Classification systems and selection systems: The risks of radical
  innovation and category spanning.
\newblock {\em Scandinavian Journal of Management}, 27(3):297--306, 2011.

\end{thebibliography}
\appendix
\section{Relational semantics via dual characterization}
\label{Appendix}

The dual correspondence between perfect lattices and RS-polarities serves as a base to generalize the Kripkean semantics of modal logic to logics with possibly non-distributive propositional base. Analogous to the dual correspondence  between Kripke frames and complete and atomic Boolean algebras with operators, one would want a dual correspondence between perfect normal lattice expansions and RS-polarities endowed with additional relations. In \cite[Section 2]{ConPal13}, a method for computing the definition of the relations dually corresponding to normal modal operators was discussed and illustrated for a certain modal signature consisting of unary and binary modal operators.

In this subsection we will report on this method, for an expansion $\mathcal{L}$  of the basic lattice language  with a unary box-modality, canonically interpreted on lattices endowed with a completely meet-preserving operation. Moreover, we will derive, by means of a dual characterization argument, its interpretation on expanded RS-polarities.

We take the connection between the satisfaction relation $\Vdash$ in Kripke frames and the interpretation of modal formulas in BAOs as our guideline: let $\mathbb{F} = (W, R)$ be a Kripke frame. From the satisfaction relation $\Vdash \ \subseteq \ W\times \mathcal{L}$ between states of $\mathbb{F}$ and formulas, an interpretation $\overline{v}: \mathcal{L}\to \mathbb{F}^+$ into the complex algebra of $\mathbb{F}$ can be defined, which is an $\mathcal{L}$-homomorphism, and is obtained as the unique homomorphic extension of the equivalent functional representation of the relation $\Vdash$ as a map $v: \mathsf{PROP}\to \mathbb{F}^+$, defined as $v(p) = {\Vdash}^{-1}[p]$\footnote{Notice that in order for this equivalent functional representation to be well defined, we need to assume that the relation $\Vdash$ is $\mathbb{F}^+$-{\em compatible}, i.e.\ that ${\Vdash}^{-1}[p]\in \mathbb{F}^+$ for every $p\in \mathsf{PROP}$. In the Boolean case, every relation from $W$ to $\mathrm{LML}$ is clearly $\mathbb{F}^+$-compatible, but already in the distributive case this is not so: indeed ${\Vdash}^{-1}[p]$ needs to be an upward- or downward-closed subset of $\mathbb{F}$. This gives rise to the persistency condition, e.g.\ in the relational semantics of intuitionistic logic.}. In this way, interpretations can be derived from satisfaction relations, so that for any $a\in \jty(\mathbb{F}^+)$ and any formula $\phi$,
\begin{equation}\label{eq: desiderata satisfaction}
a\Vdash \phi\quad \mbox{ iff }\quad a\leq \overline{v}(\phi),
\end{equation}
\noindent where, on the left-hand side, $a\in \jty(\mathbb{F}^+)$ is identified with a state of $\mathbb{F}$ via the isomorphism $\mathbb{F}\cong (\mathbb{F}^+)_+$. %As discussed in \cite{Ge06}, the semantic setting of perfect lattices   generalizes  the corresponding Boolean and distributive settings, but this time it is {\em primary}, and not {\em derived} from a pre-existing relational semantics. %To define the RS-semantics of LML starting from the algebraic semantics, we will proceed as follows: l
Conversely, consider a perfect lattice with completely meet-preserving operation $\mathbb{C} = (\mathbb{L}, \Box)$, and a homomorphic assignment $\overline{v}: \mathcal{L}\to \mathbb{C}$, and recall that  the complete lattice $\mathbb{L}$ can be identified  with the lattice $\mathbb{P}^+$ arising from some RS-polarity $\mathbb{P} = (A, X, \perp)$. We want to define a suitable relation $R = R_{\Box}$ and satisfaction relation $\Vdash_{\overline{v}}$ satisfying the condition \eqref{eq: desiderata satisfaction}. % complex algebra $\mathbb{F}^+$ of some relational structure $\mathbb{F} = (A, X, \perp, R_{\Box})$ based on the RS-polarity $(X, Y, R)$ with $R_{\Box}$ a binary relation used to interpret $\Box$, and given an interpretation $v: \mathcal{L}\to \mathbb{F}^+$, the satisfaction relation $\Vdash$ on $\mathbb{F}$ we aim to define should satisfy the condition \eqref{eq: desiderata satisfaction}.
The  method we are going to illustrate hinges on the  dual characterization of  $\overline{v}$ as a pair of relations $(\Vdash_{\overline{v}}, \succ_{\overline{v}})$ such that $\Vdash_{\overline{v}}\ \subseteq\  \jty(\mathbb{L})\times \mathcal{L}\cong A\times \mathcal{L}$ and  $\succ_{\overline{v}}\ \subseteq\ \mty(\mathbb{L})\times \mathcal{L}\cong X\times \mathcal{L}$.
This dual characterization is established by induction on formulas.

The base of the induction is clear: for every $a\in \jty(\mathbb{P}^+)$ and every $p\in \mathsf{PROP} \cup \{0, 1 \}$, we define \begin{equation}\label{eq: basic case} a\Vdash_{\overline{v}} p\quad \mbox{ iff }\quad a\leq \overline{v}(p).\end{equation}

Now let us turn to the inductive step for the box. Since $\overline{v}: \mathcal{L}\to \mathbb{P}^+$ is a  homomorphism,  $\overline{v}(\Box \phi) = \Box^{\mathbb{P}^+}\overline{v}(\phi)$. Suppose that (\ref{eq: desiderata satisfaction}) holds for $\phi$.

Since $\mathbb{P}^+$ is perfect, $\overline{v}(\phi) = \bigwedge \{x \in \mty(\mathbb{L}) \mid \overline{v}(\phi) \leq x\}$. Thus,

\begin{center}
\begin{tabular}{r c l l}
$a\leq \overline{v}(\Box \phi)$& iff &$a\leq \Box^{\mathbb{P}^+} \overline{v}(\phi)$\\
&iff & $a \leq \Box^{\mathbb{P}^+} \bigwedge \{x \in \mty(\mathbb{P}^+) \mid \overline{v}(\phi) \leq x\}$ &\\
& iff  &$a \leq \bigwedge \{\Box^{\mathbb{P}^+} x  \mid x \in \mty(\mathbb{P}^+) \mbox{ and }\overline{v}(\phi) \leq x\}$ &\\
& iff  &$\forall x[(x \in \mty(\mathbb{L}) \ \&\  \overline{v}(\phi) \leq x) \rightarrow a \leq \Box^{\mathbb{P}^+} x]$.\\
\end{tabular}
\end{center}

Notice that, at the end of this chain of equivalence, we have equivalently reduced the whole information on $\Box$ to the information whether $a \leq \Box^{\mathbb{P}^+} x$ for each $a$ and $x$. So this can be taken as the definition of the relation $R\subseteq A\times X$: we let $a Rx$ iff $a \leq \Box^{\mathbb{P}^+} y$.

To turn the last clause above  into a satisfaction clause for $\Box$, we firstly replace $\mty(\mathbb{L})$ with $X$, which we identify via the isomorphism $\mathbb{P}\cong (\mathbb{P}^+)_+$. Secondly, %nosince the relationship between the relation $R_{\Box}$ and $\Box^{\mathbb{F}^+}$ has been left completely unspecified---indeed, part of the purpose of the dual characterization is to discover this relationship---we declare that $a R_{\Box}x$ iff $a \leq \Box^{\mathbb{P}^+} y$ .
%The third and last ingredient
we need to recall the  second relation $\succ_{\overline{v}}$ between elements of $X$ and formulas, obeying the following condition, which is to be defined by induction on the structure of the formulas in such a way that the following condition holds,  analogously to (\ref{eq: desiderata satisfaction}):
\begin{equation}\label{eq:desiderata:co:sat}
x \succ \phi \quad \mbox{ iff }\quad \overline{v}(\phi) \leq x.
\end{equation}
These  considerations produce the following satisfaction clause for $\Box$:
\begin{center}
\begin{tabular}{c c c l l}
$a\Vdash_{\overline{v}}\Box \phi $& iff  &$a\leq \overline{v}(\Box \phi)$& iff  &$\forall x[(x \in X \ \&\ x\succ \phi) \rightarrow a R_{\Box}x]$\\
\end{tabular}
\end{center}

The co-satisfaction relation $\succ$ deserves some further comment: in the Boolean and distributive settings, $\succ$ is completely determined by $\Vdash$, and is hence not mentioned explicitly there. Here, in the non-distributive setting, the relation needs to be defined along with $\Vdash$. %we define $\succ$ simultaneously with  $\Vdash$
%by induction on formulas.
Equation (\ref{eq:desiderata:co:sat}) determines the base case:
\begin{equation}\label{eq:co-satisf basic case} y\succ \overline{v}(p)\quad \mbox{ iff }\quad \overline{v}(p)\leq y.\end{equation}

Specializing the clause above to powerset algebras $\mathcal{P}(W)$, we would have $y \succ_{V} p$ iff $V(p) \leq y$ iff $V(p) \subseteq W/\{ x \}$ for some $x \in W$ iff $\{ x\} \not\subseteq V(p)$ iff $x \notin V(p)$ iff $x \not \Vdash p$, which shows that  the relation $\succ$ can be regarded as an upside-down description of the satisfaction relation $\Vdash$, namely a {\em co-satisfaction}, or {\em refutation}.

The inductive step for the derivation of the co-satisfaction clause for $\Box$ goes as follows:

\begin{center}
\begin{tabular}{r c l l}
$\overline{v}(\Box \phi) \leq x$ & iff &$\bigvee \{ a \in \jty(\mathbb{L}) \mid a \leq \overline{v}(\Box \phi)\} \leq x$\\
& iff  & $\forall a[(a \in \jty(\mathbb{L})\ \&\ a \leq  \overline{v}(\Box \phi)) \rightarrow a\leq x]$\\
& iff  & $\forall a[(a \in A\ \&\ a \Vdash \Box \phi) \rightarrow a \perp x].$
\end{tabular}
\end{center}
The last line follows from equation (\ref{eq: desiderata satisfaction}) for $\Box \phi$, and the identification, via the isomorphism $\mathbb{P}\cong (\mathbb{P}^+)_+$, of $\jty(\mathbb{L})$ with $A$, and of  the lattice order $\leq$ (restricted to $\jty(\mathbb{L}) \times \mty(\mathbb{L})$) with the incidence relation $\perp$ of the polarity.

%Similar dual characterizations can also be performed to derive the interpretation of conjunction and disjunction and every other normal connectives in the language.

\section{Proofs of technical lemmas}\label{Appendix:TechLemmas}
\begin{proof} (Lemma \ref{Lemma:Galois:Stable}) We only prove the part concerning $z$. Let $x\in z{\uparrow}^{\downarrow\uparrow}$, and let us show that  $z\leq x$. That is,  let us fix $a$ such that $a\perp z$, and  show that $a\perp x$.
Since $\perp\circ \leq\  \subseteq \ \perp$,  from $a\perp z$ it follows that $\forall y (z\leq y\rightarrow a\perp y)$, which means that $a\in z{\uparrow}^{\downarrow}$. Since by assumption $x\in z{\uparrow}^{\downarrow\uparrow}$, this implies that $a\perp x$, as required.
\end{proof}

\begin{proof} (Corollary \ref{Cor:Up:Down})
Since $z{\uparrow}$ is Galois-stable and contains $z$ and, by definition, $\du{z}$ is the smallest such set, $\du{z}\subseteq z{\uparrow}$. For the converse inclusion, let $z\leq y$  and $a\perp z$. As $\perp\circ \leq\ \subseteq\ \perp$, this implies $a\perp y$, which shows that $y\in \du{z}$, as required.
\end{proof}

\begin{proof} (Lemma \ref{lemma: R circ leq subseteq R})
Assume that $aRz$ and  $z\leq y$. To show that $y\in R[a]$, by the second compatibility condition, it is enough to show that $y\in \du{R[a]}$. That is, let us fix $b\in \down{R[a]}$ and show that $b\perp y$. From  $b\in \down{R[a]}$ and $aRz$ it follows that $b\perp z$. This and $z\leq y$ imply that $b\perp y$, given that $\perp\ \circ\ \leq\ \subseteq\ \perp$. The remaining part is proven similarly.
\end{proof}

\begin{proof}(Lemma \ref{lemma:properties of C})
Clearly, $C(u)\leq 1u\leq u$, which proves the first inequality.
\begin{center}
$C(C(u)) = \bigwedge_{s\in S}sC(u) = \bigwedge_{s\in S}s(\bigwedge_{t\in S}tu) = \bigwedge_{s\in S}\bigwedge_{t\in S}stu \geq \bigwedge_{s'\in S}s'u = C(u)$.
\end{center}
\end{proof}

%% Bibliography
%% Make sure to use the bibliographystyle aiml16.

\end{document}